\documentclass[12pt,a4paper,oneside]{article}
\usepackage{amsfonts, amsmath, amssymb,latexsym, amsthm}
\usepackage{epsfig}
\usepackage{color}
\usepackage[all]{xy}

\usepackage[latin1]{inputenc}

\newtheorem{thm}{Theorem}[section]
\newtheorem{lemma}[thm]{Lemma}

\newtheorem{defi}[thm]{Definition}
\newtheorem{cor}[thm]{Corollary}

\usepackage[%ps2pdf=true,
colorlinks]{hyperref}

\usepackage{memhfixc} % Solve problems between memoir and hyperref
\usepackage%[%figure,tableall]
{hypcap} % Correct a problem with hyperref
\hypersetup{
    bookmarksnumbered,
    pdfstartview={FitH},
    citecolor={blue},
    linkcolor={red},
    urlcolor={red},
    pdfpagemode={UseOutlines}
}

\usepackage{fancyhdr}

\title{Characterization of $C^{(n)}$}
\author{Meritxell S\'aez}
\date{}
\begin{document}
\bibstyle{plain}

\maketitle

\begin{abstract}
In this paper a new geometric characterization of the $n$th symmetric product of a curve is given. Specifically, we assume that there exists a chain of smooth subvarieties $V_i$ of dimension $i$, such that $V_i$ is an ample divisor in $V_{i+1}$ and its intersection product with $V_1$ is one. That the Albanese dimension of $V_2$ is $2$ and the genus of $V_1$ is equal to the irregularity of the variety. We prove that in this case the variety is isomorphic to the symmetric product of a curve.

{\parskip7pt
\noindent \textbf{Keywords}: Symmetric product, curve, irregular variety.}
\end{abstract}

\let\thefootnote\relax\footnote{The author has been partially supported by the Proyecto de Investigaci\'on MTM2012-38122-C03-02.}
\let\thefootnote\relax\footnote{2010MSC Codes: 14J30, 14J35, 14J40}
\let\thefootnote\relax\footnote{University of Copenhagen. Contact e-mail: meritxell@math.ku.dk. Phone: +4535337598}

\section{Introduction}

The aim of this paper is to give a new characterization of the $n$th symmetric product of a curve. Following the ideas introduced in the articles \cite{CCM} and \cite{MPP1} we prove a characterization of the $n$th symmetric product of a curve by the existence of a chain of subvarieties with certain properties. This generalizes the $2$-dimensional case proved in the mentioned references. 

Let $C$ be a smooth complex projective curve of genus $g$. For an integer $n\geq 1$, the $n$th symmetric product of $C$ is the quotient of the Cartesian product by the action of the $n$th symmetric group. The action of $S_n$ on $C\times \cdots \times C$ is by permutation of the factors. It is well known that $C^{(n)}$ is a smooth and projective variety of dimension $n$ which parametrizes the effective degree $n$ divisors on $C$. Equivalently, it parametrizes the unordered $n$-tuples of points of $C$. 

Symmetric products of curves play a very important role both in the theory of algebraic curves and in the theory of higher dimensional algebraic varieties. In the first topic, they are exploited by Brill-Noether theory to study special divisors on curves. Moreover, the $i$th symmetric product determines the cuve $C$. In the second topic, they are particularly simple examples of irregular varieties in any dimension.

The square symmetric product of a curve can be described in a very precise geometric way, simply by the existence of a divisor with certain numerical properties. 

\begin{thm}(\cite{MPP1}) \label{mpp}
Let $S$ be a smooth surface of general type with irregularity $q$ containing a $1$-connected divisor $D$ such that $p_a(D)=q$ and $D^2>0$. Then the minimal model of $S$ is either
\begin{enumerate}
\item the product of two curves of genus $g_1,\,g_2\geq 2$ ($g_1+g_2=q$) or
\item the symmetric product $C^{(2)}$, where $C$ is a smooth curve of genus $q$, and $C^2=1$.
\end{enumerate}
Furthermore, if $D$ is $2$-connected, only the second case occurs.
\end{thm}
 
We remark that in the proof of this theorem the authors use the characterization of $C^{(2)}$ given in \cite{CCM}. Namely, that $C^{(2)}$ is the only minimal algebraic surface with irregularity $q$ that is covered by curves of genus $q$ and self-intersection $1$. These are the coordinate curves $C_{2,P}$, $P\in C$, that parametrize the degree two divisors in $C$ which contain the point $P$. 

In general, given a point $P\in C$, we define the divisor $C_{n,P}$ of $C^{(n)}$ as 
\[
C_{n,P}=\left\{P+\mathcal{Q}\ |\ \mathcal{Q}\in C^{(n-1)}\right\}.
\]
That is, $C_{n,P}$ is the image of the inclusion map $i_P: C^{(n-1)}\rightarrow C^{(n)}$ with $i_P(\mathcal{Q})=P+\mathcal{Q}$. The divisor $C_{n,P}$ is ample in $C^{(n)}$ (see \cite[Pag. 247]{Poli}) and isomorphic to $C^{(n-1)}$.

The numerical equivalence class of $C_{n,P}$ is independent of $P$, and so, when talking about numerical classes, the subindex $P$ will not be significant. We will call these divisors the \textbf{coordinate divisors}. When $n=2$, they are the usual \textbf{coordinate curves} in $C^{(2)}$. These coordinate divisors form a $1$-dimensional family, $\mathcal{C}$, of algebraically equivalent divisors in $C^{(n)}$ (not linearly equivalent). Moreover, its numerical class determines the family. That is, if an effective divisor of $C^{(n)}$ is numerically equivalent to $C_{n,P}$ then it belongs to the family $\mathcal{C}$ (see \cite{CS}).

The main result in this paper is the following theorem characterizing symmetric products of curves:

\begin{thm}\label{charCn}
Let $X$ be a smooth projective variety of dimension $n$. Assume that there exists a chain of inclusions 
\[
X=V_n\supset V_{n-1} \supset \cdots \supset V_2 \supset V_1=C
\] 
such that 
\begin{enumerate}
\item $V_i$ is a smooth irreducible variety with $dim(V_i)=i$.
\item For $i<n$, $V_i$ is an ample divisor in $V_{i+1}$.
\item $V_i\cdot C=1$ inside $V_{i+1}$.
\item The Albanese dimension of $V_2$ is $2$.
\item $q(X)=g(C)$. 
\end{enumerate}

Then $X\cong C^{(n)}$. Moreover, $V_i\cong C^{(i)}$ and it is a coordinate divisor inside $V_{i+1}$ for $i<n$.
\end{thm}

We prove this result by induction on the dimension of the variety. The $2$-dimensional step is a consequence of Theorem \ref{mpp}. To prove the induction step, we observe first that the $Pic^0$ varieties of the elements in the chain are isomorphic. Using these isomorphisms and generic vanishing results, we find a one dimensional algebraic family of dicisors which are birational to the $C^{(n-1)}$. This family allows us to construct a birational map between our variety and $C^{(n)}$. The image of this family by the morphism is the family $\mathcal{C}$ of the coordinate divisors. Finally, we deduce that the map is an isomorphism.  

\noindent \textbf{Acknowledgments} The author thanks Miguel Angel Barja and Joan Carles Naranjo for the multiple discussions and invaluable help on the development of this article. Many thanks also to the Universitat de Barcelona for the pre-doctoral grant APIF and their hospitality afterwards.

\textbf{Notation:} We work over the complex numbers. All varieties considered are projective and irreducible. For a smooth variety $X$ we denote by $q(X)=h^0(X, \Omega^1_X)$ its irregularity. The Albanese dimension of a variety is the dimension of its image by the Albanese morphism.

\section{Proof of the main theorem}

First, we remind some results that are useful for the proof of Theorem \ref{charCn}.

\begin{lemma}\label{injrest}
Let $X$ be an algebraic variety of dimension $n\geq 3$ and let $D$ be an ample effective reduced divisor. Then, the restriction map $Pic^0(X)\rightarrow Pic^0(D)$ is injective.
\end{lemma} 

\begin{proof}
By the Lefschetz Theorem for Picard Groups (\cite{Laz}), we have that the restriction morphism $Pic^0(X)\rightarrow Pic^0(D)$ has trivial kernel. 
\end{proof} 

We remind some results on generic vanishing theory. The main objects of interest are the cohomological support loci. 

\begin{defi}
Let $X$ be an irregular (smooth) variety of dimension $d$. The cohomological support loci of $\mathcal{O}_X$ are the algebraic sets
\[
V^i(X)=V^i(X, \mathcal{O}_X)=\{\eta\in Pic^0(X)\ |\ h^i (X, \mathcal{O}_X \otimes \eta)\neq 0\},
\]
where $i = 1,\dots, d$.
\end{defi}

The main result about the structure of the cohomological support loci was proved by Green and Lazarsfeld, with an important addition due to Simpson (the fact that the translations are given by torsion elements).
\begin{thm}\label{gv1}(\cite{GL1},\cite{GL2}, \cite{Simp})
Let $X$ be an irregular variety of dimension $d$, then $V^i(X)$ is formed by translates of subtorus of $Pic^0(X)$ by torsion elements. Moreover, 
\[
codim_{Pic^0(X)} V^i(X)\geq dim\,a (X)-i
\]
where $a(X)$ is the image of $X$ by its Albanese morphism.

In particular, $h^i (X,L) = 0$ for general $L\in Pic^0(X)$ and $i < dim\,a(X)$.
\end{thm}

\vspace{5pt}
We define the index of a family of divisors:

\begin{defi}
Given an irreducible family $\mathcal{D}\subset B \times X$, with dimension $1$ ($dimB=1$), of effective divisors in a projective variety $X$, the \textbf{index} $i=i(\mathcal{D})$ of $\mathcal{D}$ is the degree of the projection, $p_2: \mathcal{D} \rightarrow X$. Equivalently, it is the number of divisors of $\mathcal{D}$ containing the general point of $X$.
\end{defi}

Notice that the family of coordinate divisors in $C^{(n)}$ has index $n$.
\vspace{5pt}

Now, we have all the necessary tools to prove our main theorem.

\begin{proof}[Proof of Theorem \ref{charCn}]
We prove the theorem by induction. First, we observe that since $C\subset V_2$ is an irreducible smooth curve it is $2$-connected. Moreover, its self-intersection is one and hence, following the proof of Theorem \ref{mpp}, we deduce that $S:=V_2$ is birational to $C^{(2)}$. Furthermore, since the divisor $C$ is ample in $S$, in fact $S\cong C^{(2)}$, because any exceptional divisor would have intersection product $0$ with $C$. Hence, the case $n=2$ is already known. By the proof of Theorem \ref{mpp} (see \cite{MPP1}, Proposition 4.3) we have that there exists a $1$-dimensional family in $Pic^0(S)$
\[
\mathcal{W}:=\{\tilde{\eta} \in Pic^0(S)\ |\ h^0(S, \mathcal{O}_S(C)\otimes\tilde{\eta})=1\}.
\]  
It is the image of   
\[
W_1(C)=\{\eta\in Pic^0(C)\ |\ h^0(C,\mathcal{O}_C(C)\otimes \eta)=1\}
\] 
by the isomorphism $Pic^0(S)\cong Pic^0(C)$ given by the restriction map. That is, we consider $W_1(C)$ as the image of $C$ by the natural map $C\rightarrow Pic^0(C)$ defined as $p\rightarrow \mathcal{O}_C(p-C|_C)$.

Furthermore, $\mathcal{C}=\{C_{\eta},\ \eta\in \mathcal{W}\}$ is the family of coordinate curves in $C^{(2)}$, where $C_{\eta}$ is the curve such that $\mathcal{O}_S(C_{\eta})=\mathcal{O}_S(C)\otimes \eta$.

We observe that since $V_1=C$ is algebraically equivalent to a coordinate curve in $V_2\cong C^{(2)}$, it is, in fact, a coordinate curve and thus $0\in \mathcal{W}$. 

We assume now that $n\geq 3$ and that the result is proven for all $dim(X)\leq n-1$. We are going to prove the theorem for $dim(X)=n$. We consider $S:=V_2$ for the inductive process.

Since $V_i$ is ample in $V_{i+1}$ and $q(X)=g(C)$, by Lemma \ref{injrest} we obtain the following chain of isomorphisms given by the restriction maps:
\[
Pic^0(X)\cong Pic^0(V_{n-1})\cong \dots \cong Pic^0(V_i)\cong \cdots \cong Pic^0(S)\cong Pic^0(C).
\]

We add the following statement to the inductive process:

\begin{itemize}
\item  The image of $\mathcal{W}$ in $Pic^0(V_i)$ by this chain of isomorphisms parametrizes the family of coordinate divisors in $V_{i}\cong C^{(i)}$ for $i<n$. 
\end{itemize}

We remind that by the induction hypothesis, $V_{i-1}$ is a coordinate divisor in $V_i\cong C^{(i)}$ for all $i<n$. 

In what follows we denote by $N$ the divisor associated to the line bundle $\mathcal{O}_{V_{n-1}}(V_{n-1}|_{V_{n-1}})$.

\textbf{Claim:} \emph{There exists $\alpha \in Pic^0(X)$ such that 
\[
\alpha|_{V_{n-1}}=\mathcal{O}_{V_{n-1}}(V_{n-2}-N)
\]
and $V_{n-1}^n=1$.}

Consider $\mathcal{O}_S(V_{n-2}|_S-N|_S)$. We observe that $V_{n-1}\cong C^{(n-1)}$ and $S\cong C^{(2)}$ with the inclusion $S\hookrightarrow V_{n-2}$ given by a point in $C^{(n-4)}$ (when $n=3$, $V_{n-2}$ is just $C$). Therefore, $V_{n-2}|_S$ is algebraically a coordinate curve $C_{2,Q}$ in $S\cong C^{(2)}$.

Moreover, $N|_S\cdot C=V_{n-1}|_S \cdot C=1$ and hence
\[
(V_{n-2}|_S-N|_S)\cdot C=(C_{2,Q}-V_{n-1}|_S)\cdot C=0
\]
and 
\[
\begin{array}{c}
(V_{n-2}|_S-N|_S)^2=(C_{2,Q}-V_{n-1}|_S)^2=\\
C_{2,Q}^2-2C_{2,Q}\cdot V_{n-1}|_S + (V_{n-1}|_S)^2 =-1+(V_{n-1}|_S)^2\geq 0
\end{array}
\]
because $V_{n-1}$ is ample in $V_n$.

Since $C$ is ample in $S$, by the Hodge index Theorem, we deduce that $V_{n-2}|_S-N|_S$ is numerically trivial. In fact, it is algebraically trivial, because there is no torsion in $H^2(C^{(2)}, \mathbb{Z})$ (see \cite{MD}). 

By the Lefschetz Theorem for Picard Groups applied to the chain of $V_i$'s we have that the restriction map gives an injective morphism $Pic(V_{n-1}) \hookrightarrow Pic(S)$. Then, from the isomorphism $Pic^0(V_{n-1})\cong Pic^0(S)$ and 
\[
\mathcal{O}_S(V_{n-2}|_S-N|_S)\in Pic^0(S)
\] 
we deduce that 
\[
\mathcal{O}_{V_{n-1}}(V_{n-2}-N)\in Pic^0(V_{n-1}).
\]
Consequently, by the isomorphism between the $Pic^0$'s, there exists an $\alpha$ as claimed. 

Finally, since $V_{n-2}$ and $N$ are numerically equivalent, we obtain that $1=V_{n-2}^{n-1}=N^{n-1}=(V_{n-1}|_{V_{n-1}})^{n-1}=V_{n-1}^n$. $\Diamond$ 

Now, consider the exact sequence
\begin{equation}\label{exactsequence}
0\rightarrow \mathcal{O}_{X} \rightarrow \mathcal{O}_{X}({V_{n-1}}) \rightarrow \mathcal{O}_{V_{n-1}}({V_{n-1}}) \rightarrow 0.
\end{equation}

Let $\mathcal{W}_n$ be the image of $\mathcal{W}$ by the isomorphism $Pic^0(V_{n-1})\cong Pic^0(S)$ and $\eta \in \mathcal{W}_n$ a general element. We tensor \eqref{exactsequence} with $\alpha\otimes \eta$ and get
\[
0 \rightarrow  \alpha\otimes \eta \rightarrow \alpha\otimes \eta \otimes \mathcal{O}_{X}({V_{n-1}}) \rightarrow \mathcal{O}_{V_{n-1}}(V_{n-2})\otimes\eta|_{V_{n-1}} \rightarrow 0.
\]
We take cohomology and obtain 
\[
\begin{array}{ll}
0 \rightarrow  & H^0({X},\alpha\otimes \eta) \rightarrow H^0({X},\alpha\otimes \eta \otimes \mathcal{O}_{X}({V_{n-1}})) \rightarrow \\[4pt]
&H^0({V_{n-1}}, \mathcal{O}_{V_{n-1}}(V_{n-2})\otimes\eta|_{V_{n-1}}) \rightarrow H^1({X}, \alpha\otimes \eta) \rightarrow \dots
\end{array}
\]

First of all, we observe that $H^0({X},\alpha\otimes \eta)=0$ since $\alpha\otimes \eta \in Pic^0({X})$ is non trivial. 

Second, we notice that the image of the Albanese morphism of $X$ has dimension greater or equal than two. Indeed, we know that the image of the Albanese morphism of $S\cong C^{(2)}$ is a two dimensional subvarity of $Alb(S)=J(C)$. By the identification of $Pic^{0}$'s, this subvariety of $J(C)$ lives inside the image of the Albanese morphism of $X$, hence, it is of dimension at least two. We can apply generic vanishing results and deduce that $V^1(X)=\{\varsigma\in Pic^0(X)\ |\ h^1(X, \varsigma)>0\}$ is the union of finitely many translates of proper abelian subvarieties.

Furthermore, we know that $W_1(C)$ generates $Pic^0(C)$, hence, its image by the identification $Pic^0(X)\cong Pic^0(C)$ generates $Pic^0(X)$. When we translate it by a fixed element $\alpha\in Pic^0(X)$ it still generates, so by the generic vanishing results, it cannot be contained in $V^1(X)$. Hence, for a general $\eta\in \mathcal{W}_n$ we obtain that $\alpha \otimes \eta \notin V^1(X)$ and thus $H^1(X, \alpha \otimes \eta)=0$. 

Therefore, for $\eta\in \mathcal{W}_n$ general we have that  
\[
h^0(X,\alpha\otimes \eta \otimes \mathcal{O}_X(V_{n-1}))=h^0(V_{n-1}, \mathcal{O}_{V_{n-1}}(V_{n-2})\otimes\eta|_{V_{n-1}})=1>0.
\] 
And by semicontinuity, $h^0(X, \alpha \otimes \eta \otimes \mathcal{O}_{X}({V_{n-1}}))>0$ for all $\eta\in \mathcal{W}_n$.

Thus, we have a $1$-dimensional family, $\mathcal{D}$ in ${X}$, of effective divisors algebraically equivalent to ${V_{n-1}}$. Let $H_{\eta}$ denote the effective divisor in $X$ such that $\mathcal{O}_X(H_{\eta})=\mathcal{O}_X(V_{n-1})\otimes \alpha \otimes \eta$. We observe that: 
\begin{itemize}
\item $V_{n-1}\cdot H_{\eta} =(V_{n-2})_{\eta}$. Indeed, 
\[
\mathcal{O}_{V_{n-1}}(H_{\eta})=\mathcal{O}_{V_{n-1}}(N)\otimes \alpha|_{V_{n-1}}\otimes \eta=\mathcal{O}_{V_{n-1}}({V_{n-2}})\otimes \eta
\] 
where we consider $\eta \in Pic^0(V_{n})$ or $Pic^0(V_{n-1})$ indistinctively by the isomorphism given by the restriction map.

\item Since $H_{\eta}$ is algebraically equivalent to $V_{n-2}$, we have that $H_{\eta}$ is ample and $H_{\eta}^n=1$. Hence, $Pic^0(X) \cong Pic^0(H_{\eta})\cong Pic^0(C)$. In particular, when $H_{\eta}$ is smooth, $q(H_{\eta})=g(C)$. 

\item If $H_{\eta}$ is smooth, since $V_{n-1}\cdot H_{\eta} =(V_{n-2})_{\eta}\cong C^{(n-2)}$, we can apply the induction hypothesis to $H_{\eta}$ and deduce that $H_{\eta} \cong C^{(n-1)}$. In addition we obtain that in $Pic^0(H_{\eta})$ there is a $1$-dimensional family $\{\varsigma \in Pic^0(H_{\eta})\ |\ h^0(H_{\eta}, \mathcal{O}_{H_{\eta}}((V_{n-2})_{\eta})\otimes \varsigma)>0\}$ which is the image of $\mathcal{W}$ via the identification $Pic^0(H_{\eta})\cong Pic^0(C)$. 
\end{itemize} 

Assume for a moment that $H_{\eta}$ is smooth for $\eta$ generic.

Since $C^{(i)}$ deforms in an algebraic family only as the $i$th symmetric product of a curve, we deduce that the general element in $\mathcal{D}$ is birational to $C^{(n-1)}$. Moreover, since $V_{n-1}\cdot H_{\eta} = (C^{n-2})_{\eta}$ we obtain that the restriction of $\mathcal{D}$ to a general divisor in the family is the family of coordinate divisors in $H_{\eta}\approx C^{(n-1)}$. 

Consequently, since the index of the family of coordinate divisors on $C^{(n-1)}$ is $n-1$, we deduce that the index of $\mathcal{D}$ in $X$ is $n$. Indeed, given a general point in $H_{\eta}$, we have $n-1$ other elements of $\mathcal{D}$ containing it, that together with $H_{\eta}$ are a total of $n$ elements of the family.

Next, we see that indeed $X\cong C^{(n)}$.  

Let $Q\in X$ be a general point and let $H_1, \dots, H_n$ be the divisors in $\mathcal{D}$ containing $Q$. Let $D_1=V_{n-1}\cdot H_1$, then $D_1$ is a coordinate divisor in $V_{n-1}\cong C^{(n-1)}$, hence, it is of the form $C^{(n-2)}+{P_1}$, for certain $P_1\in C$. In a similar way, $H_i\cdot V_{n-1}=C^{(n-2)}+{P_i}$. Thus, we have a birational map
\[
\begin{matrix}
X & \dashrightarrow & C^{(n)} \\
Q & \dashrightarrow & P_1+\dots+ P_n.
\end{matrix} 
\]

In fact, $X\cong C^{(n)}$. Indeed, any curve contracted by the birational map would have product $0$ with $V_{n-1}$, which is not possible since $V_{n-1}$ is ample in $X$. Observe finally that if $V_{n-1}\cdot H_{\eta}= C^{(n-2)}+{P}$, then $H_{\eta}=C_{n,P}$, the coordinate divisor with base point $P$, and hence, $\mathcal{W}_n$ parametrizes the coordinate divisors in $C^{(n)}$.

Finally, we study the possible singularities of the hypersurfaces $H_{\eta}$ to prove that indeed the general one is smooth. 

First, a divisor $H_{\eta}$ does not contain a curve of singularities. Otherwise, since $V_{n-1}$ is ample, this curve would cut $V_{n-1}$ in a point, and then $(C^{(n-2)})_{\eta}$ should be singular, contradicting our hypothesis. Hence, each $H_{\eta}$ has at most a finite number of singularities.

Second, the possible singularities do not deform with the divisors in the family. Otherwise, there would be some curves $\{B_i\}$ such that the intersection point of $B_i$ and $H_{\eta}$ would be a singular point of $H_{\eta}$. Since $V_{n-1}$ is ample, a curve $B_i$ would intersect $V_{n-1}$ in a point $P\in (C^{(n-2)})_{\eta}$ for a certain $\eta$, and then $(C^{(n-2)})_{\eta}$ should be singular.

Third, there is no base curve for the family $\mathcal{D}$. Otherwise, this curve would intersect $V_{n-1}$ and then the family $\mathcal{D}_{n-2}$ of coordinate divisors in $C^{(n-1)}$ should have a base point.

Finally, there is no singularity $Q$ common to all $H_{\eta}\in \mathcal{D}$. Otherwise, the point $Q$ would be a base point of the family and all varieties would have multiplicity at least two at this point, therefore, $V_{n-1}^n\geq 2$, contradicting $V_{n-1}^n=1$.

Therefore, not all elements $H_{\eta}\in \mathcal{D}$ are singular, in fact, the general one is smooth, and those singular have at most isolated singularities.
\end{proof}

From the theorem we deduce the following result, less general but simpler in its hypothesis.

\begin{cor}
Let $X$ be a smooth projective variety of dimension $n$. Assume that there exists a divisor $D$ isomorphic to $C^{(n-1)}$ such that, if $C$ denotes a coordinate curve in $D\cong C^{(n-1)}$, then $D\cdot C=1$. Assume also that $q(X)=g(C)$. Then $X\cong C^{(n)}$.
\end{cor}

%\noindent \textbf{Acknowledgments}

\bibliographystyle{alpha}
\bibliography{bib}

\end{document}